\documentclass[10pt]{amsart}
\usepackage{a4wide}
\usepackage{amssymb}
\usepackage{mathrsfs}
\theoremstyle{plain}
\newtheorem*{prop*}{Proposition}

\newtheorem*{thm*}{Theorem}

\newtheorem*{cor*}{Corollary}

\newtheorem{thmintro}{Theorem}

\newtheorem{corintro}[thmintro]{Corollary}

\newtheorem*{convention*}{Convention}
\theoremstyle{definition}
\newtheorem*{defn*}{Definition}

\newtheorem*{rem*}{Remark}

\newtheorem*{scholium*}{Scholium}
\newtheorem*{question*}{Question}
\newtheorem*{ack*}{Acknowledgements}

\newtheorem*{example*}{Example}
\numberwithin{equation}{section}

\newcommand{\ro}{\varrho}

\newcommand{\RR}{\mathbf{R}}

\newcommand{\bb}{\overline{B}}
\newcommand{\bid}{^{**}}

\newcommand{\inv}{^{-1}}
\newcommand{\se}{\subseteq}

\begin{document}
\title{A fixed point theorem for $L^1$ spaces}
\author[Bader, Gelander and Monod]{U. Bader, T. Gelander and N. Monod}
\address{Technion, Israel; Hebrew University, Israel; EPFL, Switzerland}
\thanks{Supported in part by the Israel Science Foundation, the European Research Council and the Swiss National Science Foundation.}
\begin{abstract}
We prove a fixed point theorem for a family of Banach spaces including notably $L^1$ and its non-commutative analogues. Several applications are given, e.g. the optimal solution to the ``derivation problem'' studied since the 1960s.
\end{abstract}
\maketitle

\section{Introduction}

Andr\'es Navas asked us if there is a fixed point theorem for all isometries of $L^1$ that preserve a given bounded set. Unlike many known cases where a geometric argument applies, there is a fundamental obstruction in $L^1$: \itshape any infinite group $G$ admits a fixed-point-free isometric action on a bounded convex subset of $L^1$\upshape. Example: the $G$-action on the affine subspace of summable functions of sum one on $G$. This action is fixed-point-free and preserves the closed convex bounded set of non-negative functions. The obvious (and only) fixed point, zero, is outside the space.

\medskip
Thus, we have to search for fixed points possibly outside the convex set, indeed outside the affine subspace it spans. We shall do this more generally for any \emph{L-embedded} Banach space $V$, that is, a space whose bidual can be decomposed as $V\bid=V\oplus_1V_0$ for some $V_0\se V\bid$ (and $\oplus_1$ indicates that the norm is the sum of the norms on $V$ and $V_0$). Recall that $L^1$ is L-embedded by the Yosida--Hewitt decomposition~\cite{Yosida-Hewitt} and that this holds more generally for the predual of any von Neumann algebra~\cite[III.2.14]{TakesakiI}; in particular, for the dual of any C*-algebra.

\begin{thmintro}\label{thm:main}
Let $A$ be a non-empty bounded subset of an $L$-embedded Banach space $V$.

Then there is a point in $V$ fixed by every isometry of $V$ preserving $A$. Moreover, one can choose a fixed point which minimises $\sup_{a\in A}\|v-a\|$ over all $v\in V$.
\end{thmintro}

We recall that an isometric action of a group $G$ on a Banach space $V$ is given by a linear part and a cocycle $b:G\to V$. The cocycle is the orbital map of $0\in V$ and a fixed point $v$ corresponds to a trivialisation $b(g) = v - g.v$, where $g.v$ is the linear action. The above norm statement implies that one can arrange $\|v\|\leq \sup_g\|b(g)\|$ by considering $A=b(G)\ni 0$.

\medskip
As a special case (the ``commutative'' case), we recover the main theorem of~\cite{Losert08} due to Losert, but with an improved (indeed optimal) norm estimate:

\begin{corintro}[cf.~\cite{Losert08}]\label{cor:losert}
Let $G$ be a group acting by homeomorphisms on a locally compact space $X$. Then any bounded cocycle $b:G\to M(X)$ to the space of (finite Radon) measures on $X$ is trivial. More precisely, there is a measure $\mu$ with $\|\mu\|\leq \sup_{g\in G}\|b(g)\|$ such that $b(g) = \mu - g.\mu$ for all $g\in G$.\qed
\end{corintro}

\noindent
Indeed, $M(X)$ is the dual of the (commutative) C*-algebra $C_0(X)$ and hence the predual of a von Neumann algebra.

\medskip
Numerous consequences of Corollary~\ref{cor:losert} are listed in~\cite{Losert08}; let us only recall that it settles the so-called \emph{derivation problem} whose history began in the 1960's: If $G$ is a locally compact group, then any derivation from the convolution algebra $L^1(G)$ to $M(G)$ is inner. This is often phrased in terms of derivations ``of $L^1(G)$'' since any derivation $L^1(G)\to M(G)$ must range in $L^1(G)$ by Paul Cohen's factorisation theorem~\cite{Cohen59}. It also follows that any derivation of $M(G)$ is inner. Our norm estimate is stronger and in fact optimal by Remark~7.2(a) in~\cite{Losert08}.

\medskip
As observed by Uffe Haagerup, Theorem~\ref{thm:main} also yields a new
proof that all C*-algebras are weakly amenable, which was proved in~\cite{Haagerup83} using the Grothendieck--Haagerup--Pisier inequality. In fact, our theorem immediately implies that any continuous derivation from any normed algebra $A$ to a predual $M_*$ of a von Neumann algebra is inner as soon as $A$ is spanned by the elements represented as invertible isometries of $M_*$ (see the proof of the corollary below). In the particular case of C*-algebras, we obtain the following general statement.

\begin{corintro}\label{cor:uffe}
Let $A$ be a unital C*-algebra. Let $M_*$ be the predual of a von Neumann algebra. Assume $M_*$ is a Banach bimodule over $A$. Then any arbitrary derivation $D:A\to M_*$ is inner.

Moreover, we can choose $v\in M_*$ with $D(a)=v.a-a.v$ such that $\|v\| \leq \|D\|$.
\end{corintro}

Haagerup's weak amenability of $A$ is given by the special case $M_*=A^*$. Our definition of Banach bimodule demands $\|a.v.b\| \leq \|a\|\cdot \|v\|\cdot\|b\|$ ($a,b\in A$, $v\in M_*$).

\begin{proof}[Proof of Corollary~\ref{cor:uffe}]
By Theorem~2 in~\cite{Ringrose72}, $D$ is continuous; thus it is bounded (by $\|D\|<\infty$) on the group $G$ of unitaries of $A$.  The map $G\to M_*$ given by $g\mapsto D(g).g\inv$ is a cocycle for the Banach $G$-module structure defined by the rule $v\mapsto g.v.g\inv$. Theorem~~\ref{thm:main} thus yields $v$, with norm bounded by $\|D\|$, such that $D(g)=v.g - g.v$ for all $g\in G$. The statement follows since any element of $A$ is a combination of four unitaries (in fact, three~\cite{Kadison-Pedersen85}).
\end{proof}

Finally, returning to the case $V=L^1$ of Theorem~\ref{thm:main}, we consider actions without \emph{a priori} boundedness of the orbits and obtain a new characterisation of Kazhdan groups:

\begin{corintro}\label{cor:T}
Let $\Omega$ by any measure space. Then any isometric action of a Kazhdan group on $L^1(\Omega)$ has a fixed point.

Moreover, this fixed point property characterises Kazhdan groups amongst all countable (or locally compact $\sigma$-compact) groups.
\end{corintro}

By the Kakutani representation theorem~\cite{Kakutani41b}, this corollary applies unchanged to abstract $L^1$ spaces, for instance to $M(X)$ for any locally compact space $X$.

\begin{proof}[Proof of Corollary~\ref{cor:T}]
Recall that any isometric action of a Kazhdan group on an $L^1$ space has bounded orbits because of a Fock space argument  (see e.g.~\cite[1.3(2)]{Bader-Furman-Gelander-Monod}). Therefore, Theorem~\ref{thm:main} implies the first part of the statement. Conversely, let $G$ be a locally compact $\sigma$-compact group with this fixed point property. A standard argument shows that $G$ has the $L ^1$-version $(T_{L^1})$ of property~$(T)$, see~\cite[1.3(1)]{Bader-Furman-Gelander-Monod} (this argument holds in the $\sigma$-compact generality). Theorem~A in~\cite{Bader-Furman-Gelander-Monod} shows that $(T_{L^1})$ implies $(T)$; although it is claimed there for $L^p$ with $1<p<\infty$, the proof applies unchanged to $L^1$, using the Connes--Weiss construction~\cite{Connes-Weiss} (exposed also in Theorem~6.3.4 of~\cite{Bekka-Harpe-Valette}).
\end{proof}

\section{Proof of the Theorem}

We first recall the concept of Chebyshev centre. Let $A$ by a non-empty bounded subset of a metric space $V$. The \emph{circumradius} of $A$ in $V$ is
$$\ro_V(A) = \inf\big\{ r\geq 0 : \exists\,x\in V \text{ with } A\se \bb(x, r) \big\},$$
where $\bb(x, r)$ denotes the closed $r$-ball around $x$. The \emph{Chebyshev centre} of $A$ in $V$ is the (possibly empty) set
$$C_V(A) = \big\{ c\in V : A\se \bb(c, \ro_V(A)) \big\}.$$
Notice that $C_V(A)$ can be written as an intersection of closed balls as follows:
$$C_V(A) = \bigcap_{r>\ro_V(A)} C_V^r(A) \kern3mm\text{where}\kern3mm C_V^r(A) = \bigcap_{a\in A} \bb(a, r).$$
Thus, when $V$ is a normed space, $C_V(A)$ is a bounded closed convex set. More importantly, when $V$ is a dual Banach space, we deduce from Alao\u glu's theorem that $C_V(A)$ is weak-* compact and that it is non-empty because the non-empty sets $C_V^r(A)$ are monotone in $r$. (For general Banach spaces, $C_V(A)$ is very often empty, even when $A$ consists of just three points~\cite{Konyagin88},\cite{Vesely01}.)

\begin{prop*}
Let $A$ be a non-empty bounded subset of an L-embedded Banach space $V$. Then the convex set $C_V(A)$ is weakly compact and non-empty.
\end{prop*}

\begin{proof}
Consider $A$ as a subset of $V\bid$ under the canonical embedding $V\se V\bid$. In view of the above discussion, $C_{V\bid}(A)$ is a non-empty weak-* compact convex set. We claim that it lies in $V$ and coincides with $C_V(A)$; the proposition then follows. Let thus $c\in C_{V\bid}(A)$ and write $c=c_V + c_{V_0}$ according to the decomposition $V\bid=V\oplus_1 V_0$. Then, for any $a\in A$, we have
$$\| a-c \| = \|a - c_V\| + \|c_{V_0}\|$$
since $A\se V$. Therefore,
$$\ro_{V\bid}(A) = \sup_{a\in A} \| a-c \| = \sup_{a\in A} \|a - c_V\| + \|c_{V_0}\|  \geq \ro_V(A) + \|c_{V_0}\|.$$
Since $\ro_{V\bid}(A) \leq \ro_V(A)$ anyway, we deduce $c_{V_0}=0$ and  $\ro_{V\bid}(A) =\ro_V(A)$, whence the claim.
\end{proof}

\begin{proof}[Proof of Theorem~\ref{thm:main}]
Since the definition of $C_V(A)$ is metric, it is preserved by any isometry preserving $A$. By the proposition, we can apply the Ryll-Nardzewski theorem and deduce that there is a point of $C_V(A)$ fixed by all isometries preserving $A$. The norm condition follows from the definition of centres.
\end{proof}

We remind the reader that in the present context the Ryll-Nardzewski theorem has a particularly short geometric proof relying on the dentability of weakly compact sets~\cite{Asplund-Namioka}.

\section{Comments}
\noindent
\textbf{a. }In marked contrast to classical fixed point theorems, there is no hope to find a fixed point inside a general bounded closed convex subset of $L ^1$, as pointed out in the opening. As a case in point, the weak compactness of the Ryll-Nardzewski theorem is a strong restriction in $L^1$ since it imposes equi-integrability, and yet it seems almost unavoidable in light of~\cite[Thm.~4.2]{Dominguez-Japon-Prus} if one insists on the classical statement.

\bigskip
\noindent
\textbf{b. }For the proposition, a canonical norm one projection $V\bid\to V$ is not enough. Indeed, any dual space is canonically $1$-complemented in its own bidual, but the fixed point property in all duals characterises amenability. Specifically, any non-amenable group $G$ has a fixed-point-free action with bounded orbits in $(\ell^\infty(G)/\RR)^*$.

\bigskip
\noindent
\textbf{c. }It would be interesting to find a purely geometric version of the proposition, since we prove compactness out of geometric assumptions. Notice however that the compact set $C_V(A)$ might still be large. If for instance $A$ consists of just the two points $0_{[0,1]}, 1_{[0,1]}$ in $V=L^1([0,1])$, then $C_V(A)$ is the infinite-dimensional set of functions $0\leq f\leq 1$ with $\int f=1/2$.
It would further be interesting to study the dynamics of the transformation $A\mapsto C_V(A)$; for instance, it can have orbits of period~$1$ or~$2$, but no other finite period.

\begin{ack*}
We are indebted to Andr\'es Navas for asking the question about $L^1$, motivated by~\cite{Coronel-Navas-Ponce}. We thank Uffe Haagerup for pointing out the application of the general result to weak amenability. Conversations with Narutaka Ozawa were very stimulating.
\end{ack*}


\end{document}